\documentclass[12pt]{article}
\usepackage{amsmath,amssymb,color,amscd}

\def\seq#1#2#3{#1_{#2},\,\ldots,#1_{#3}}
\def\abs#1{\vert#1\vert}
\def\w{\widetilde}
\def\h{\widehat}
\def\bb{\overline}
\def\Ker{\text{Ker\,}}

\def\vv{{\underline{v}}}
\def\nuv{{\underline{\nu}}}
\def\nunu{\underline{\nu}}
\def\tt{{\underline{t}}}
\def\ww{\underline{w}}
\def\WW{\underline{W}}
\def\mm{\underline{m}}

\def\1{\underline{1}}
\def\P{\mathbb P}

\def\Z{\mathbb Z}
\def\Q{\mathbb Q}
\def\C{\mathbb C}

\def\OO{{\cal O}}
\def\XX{{\cal X}}
\def\DD{{\cal D}}

\def\oD{\stackrel{\circ}{\cal D}}
\def\oE{\stackrel{\circ}{E}}

\newtheorem{theorem}{Theorem}
\newtheorem{statement}{Statement}

\newtheorem{proposition}{Proposition}

\newenvironment{definition}
{\smallskip\noindent{\bf Definition\/}:}{\smallskip\par}
\newenvironment{example}
{\smallskip\noindent{\bf Example\/}.}{\smallskip\par}

\newenvironment{remark}
{\smallskip\noindent{\bf Remark\/}.}{\smallskip\par}

\newenvironment{proof}
{\noindent{\bf Proof\/}.}{{ $\Box$}\smallskip\par}

\title{An equivariant Poincar\'e series
of filtrations and monodromy zeta functions\footnote{Math. Subject Class. 14B05, 13A18,
16W22, 16W70.
Keywords: finite group actions, filtrations,
Poincar\'e series, monodromy zeta functions,  plane valuations.}
}

\author{
A.~Campillo,
\and F.~Delgado,\thanks{Supported by the grants
MTM2012-36917-C03-01 / 02
(both grants with the help of FEDER Program).} \and S.M.~Gusein-Zade
\thanks{
Supported by the
grants RFBR--13-01-00755, NSh--5138.2014.1.
} }

\date{}
\begin{document}
\def\eps{\varepsilon}

\maketitle

\begin{abstract}
We define a new equivariant (with respect to a finite group $G$ action) version
of the Poincar\'e series of a multi-index filtration as an element of the power
series ring ${\widetilde{A}}(G)[[t_1, \ldots, t_r]]$ for a certain modification
${\widetilde{A}}(G)$ of the Burnside ring of the group $G$. We give a formula
for this Poincar\'e series of a collection of plane valuations in terms
of a $G$-resolution of the collection. We show that, for
filtrations on the ring of germs of functions in two variables defined
by the curve valuations corresponding to the irreducible
components of a plane curve singularity defined by a $G$-invariant function germ,
in the majority of cases this equivariant Poincar\'e series determines the
corresponding equivariant monodromy zeta functions defined earlier.
\end{abstract}

\section*{Introduction}\label{sec0}
It was shown (see, e.g., \cite{IJM}) that the multi-variable Poincar\'e series
of the multi-index filtration on the ring of germs of functions in two variables defined
by the curve valuations corresponding to the irreducible components of a (reducible) plane
curve singularity $C$ coincides with the multi-variable Alexander polynomial of the
corresponding algebraic link $C\cap S^3_\eps\subset S^3_\eps$. Up to now this coincidence
has no conceptual explanation. It is obtained by direct computations of the both objects
in the same terms and comparison of the results. Generalizations of this relation (e.g.,
to an equivariant setting) can help to understand a reason for it. A possibility to try
to get an equivariant generalization of this relation was restricted by the lack of
equivariant versions of the Poincar\'e series of filtrations and of the monodromy zeta
functions.

There were several attempts to define equivariant versions of the Poincar\'e series
and of the monodromy zeta functions for a finite group $G$ action:
\cite{MMJ}, \cite{RMC}, \cite{GLM}, \cite{GLM-zeta2}.
In particular, in \cite{MMJ} an equivariant version of the Poincar\'e series
was defined as an element of $R_1(G)[[t_1, \ldots, t_r]]$, where $R_1(G)$ is the subring of the ring of representations
of the group $G$ generated by the one-dimensional representations.
This version appeared to be useful for computation of Seiberg-Witten invariants
of the links of the so-called splice-quotient surface singularities (see, e.g., \cite{Nemethi}).
However a comparison of the equivariant versions of the Poincar\'e series
and of the monodromy zeta functions
was difficult since they were elements of different rings. The
notion of the (usual, non-equivariant) Poincar\'e series is
related to the notion of integration with respect to the Euler
characteristic: see, e.g. \cite{IJM}.
An ``equivariant version''
of the ring $\Z$ of integers is the Burnside ring $A(G)$ of the group $G$, i.e.
the Grothendieck ring of finite $G$-sets (see, e.g., \cite{TtD}). In some approaches
an equivariant version of the Euler characteristic is an element of the Burnside ring.
In \cite{GLM} and \cite{GLM-zeta2} equivariant versions of the monodromy zeta function
were defined as elements of the power series rings $(A(G)\otimes \Q)[[t_1, \ldots, t_r]]$
and $A(G)[[t_1, \ldots, t_r]]$ respectively.

Here we define a new equivariant version $P_{\{\nu_i\}}^G(t_1, \ldots, t_r)$
of the Poincar\'e series as an element of $\widetilde{A}(G)[[t_1, \ldots, t_r]]$
for a certain modification $\widetilde{A}(G)$ of the Burnside ring
($\{\nu_i\}$ is a collection of order functions on the ring of
germs of analytic functions of a complex analytic space).
A reduction
of this series is an element of $A(G)[[t_1, \ldots, t_r]]$. We give a formula for the
equivariant Poincar\'e series $P_{\{\nu_i\}}^G(t_1, \ldots, t_r)$ of a collection
$\{\nu_i\}$ of plane valuations
in terms of a $G$-resolution of the collection. This formula uses a
pre-$\lambda$-structure on the ring $\widetilde{A}(G)$. We also
correct a certain inaccuracy
in the corresponding formula in \cite{RMC} (and of its use in \cite{Arkiv}).
We show that (with some exceptions) the equivariant Poincar\'e series of a collection of
curve valuations corresponding to the irreducible components of a plane curve singularity
defined by a $G$-invariant equation determines the corresponding equivariant monodromy
zeta-functions from \cite{GLM} and \cite{GLM-zeta2} by a simple algorithm.

\section{The equivariant Poincar\'e series
}\label{sec1}

We shall consider the Grothendieck ring of finite $G$-sets with an additional
structure.

\begin{definition}
A finite {\it equipped} $G$-{\it set} is a pair $\w{X}=(X, \alpha)$ where:
\begin{enumerate}
\item[$\bullet$] $X$ is a finite $G$-set;
\item[$\bullet$] $\alpha$ associates to each point $x\in X$ a
one-dimensional representation
$\alpha_x$ of the isotropy subgroup $G_x=\{a\in G : ax=x\}$ of the point $x$ so that,  for
$a\in G$, one has $\alpha_{ax}(b)=\alpha_x(a^{-1}ba)$, where
$b\in G_{ax}=a G_x a^{-1}$.
\end{enumerate}
\end{definition}

Let $\w{A}(G)$ be the Grothendieck group of finite equipped $G$-sets (with respect to
the disjoint union of equipped $G$-sets). The cartesian product $\w{X}\times\w{Y}$ of
two equipped $G$-sets $\w{X}=(X, \alpha)$ and $\w{Y}=(Y,\beta)$ is the pair
$(X\times Y, \gamma)$, where $X\times Y$ is the cartesian product of the $G$-sets $X$
and $Y$ with the diagonal $G$-action and
$\gamma_{(x,y)}(b) = \alpha_x(b)\beta_y(b)$ for $b\in G_{(x,y)}=G_{x}\cap G_y$.
The cartesian product defines a ring structure on $\w{A}(G)$.
The class of an equipped $G$-set $\w X$ in the Grothendieck ring $\w{A}(G)$ will be
denoted by $[\w{X}]$. The unit $1$ in the ring $\w{A}(G)$ is represented by the
one-point $G$-set $G/G$
with the trivial representation of (the isotropy subgroup) $G$ associated to it.
As an Abelian group $\widetilde{A}(G)$ is freely generated by the classes of the
irreducible equipped $G$-sets $[G/H]_{\alpha}$ for all the conjugacy classes $[H]$
of subgroups of $G$ and for all one-dimensional representations $\alpha$ of $H$
(a representative of the conjugacy class $[H]\in \mbox{Conjsub\,}G$).

\begin{example}
The Burnside ring $A(\Z_2)$ of the cyclic group of order $2$ is the free abelian
group generated by the classes $1=[\Z_2/\Z_2]$
 and $[\Z_2/(e)]$. The ring $\w{A}(\Z_2)$ is the free
abelian group generated the classes 1, $[\Z_2/\Z_2]_{\sigma}$ and $[\Z_2/(e)]$ where
the class $[\Z_2/\Z_2]_{\sigma}$ is represented by the one-point $\Z_2$-set
$\Z_2/\Z_2$ with the nontrivial representation of $\Z_2$ associated to the point. The
multiplication is defined by
$[\Z_2/\Z_2]_{\sigma}\cdot [\Z_2/(e)]=[\Z_2/(e)]$.
\end{example}

Let $R_1(G)$ be the subring of the ring $R(G)$ of representations
of the group $G$ generated by the one-dimensional
representations. (If the group $G$ is Abelian one has
$R_1(G)=R(G)$.)
There are natural homomorphisms from the ring $\w{A}(G)$ to the
rings $A(G)$,
$\Z$ and to $R_1(G)$ respectively. The reduction
$\rho: \w{A}(G)\to A(G)$ is defined by forgetting the representation corresponding to
the points. The reduction $\hat{\rho}:\w{A}(G)\to \Z$ is defined by forgetting the
representations and the $G$-action. (Thus one gets an element of the Grothendieck
ring of finite sets isomorphic to $\Z$.)
The homomorphism $\varepsilon: \w{A}(G)\to R_1(G)$ is
defined in the following way.
For an equipped $G$-set $\w{X}=(X,\alpha)$, let $X^G$ be the set of fixed points of the
$G$-action. For each point $x\in X^G$, $\alpha_x$ is a one-dimensional representation
of the group $G$ (coinciding with the isotropy subgroup). Thus one gets a finite set
with a one-dimensional representation of $G$ associated to each point. The
Grothendieck ring of such sets is $R_1(G)$. One define $\varepsilon([\w{X}])$ as
$[(X^G,\alpha_{\vert X^G})]\in R_1(G)$.

\begin{remark}
In \cite{RMC} there were defined a notion of a (``locally finite") $(G,r)$-set and
the corresponding Grothendieck ring $K_0((G,r)-\mbox{sets})$. One can see that a
finite
$(G,r)$-set is an equipped $G$-set with an additional structure (a $\Z^r_{\ge
0}$-valued function on it). Therefore each finite $(G,r)$-set defines a finite
equipped $G$-set.
\end{remark}

A pre-$\lambda$-structure on a ring $R$ is a map from $R$ to
$1 + tR[[t]]$ ($u\mapsto \lambda_u(t)$, $u\in R$) such that
$\lambda_{u+v}(t)=\lambda_u(t) \lambda_{v}(t)$: see e.g.
\cite{Lambda}. In what follows we shall use a
(particular) pre-$\lambda$-structure  on the ring $\w{A}(G)$. Since $\lambda_1(t)$
will be equal to $(1-t)^{-1} = 1+t + t^2 +\cdots $, we shall denote
$\lambda_u(t)$ by $(1-t)^{-u}$.
We shall define
\begin{equation}\label{lambda}
\lambda_{[\w X]}(t) =
(1-t)^{-[\w{X}]} : = 1 + [\w{X}]t + [S^2\w{X}]t^2 + \cdots\,,
\end{equation}
where $S^k\w X$ are the symmetric powers of the equipped $G$-set $\w X$ defined
below.

For an equipped $G$-set $\w{X}=(X,\alpha)$,
its symmetric power $S^k \w X$ is an equipped $G$-set described in the following
way. It is the pair $(S^k X, \alpha^{(k)})$ where $S^kX=X^k/S_k$ is the $k$-th
symmetric power of the
$G$-set $X$ (with the natural action of $G$ on it). Let an unordered collection
$\{\seq x1k\}$ represent a point of $S^kX$. One can write it as
$\{\mu_1 y_1, \ldots, \mu_s y_s\}$ where $\seq y1s$ are different points from the
collection $\seq x1k$ and $\mu_i$ are their multiplicities,
$\sum \mu_i=k$. The isotropy subgroup of
the point $\{\seq x1k\}\in S^kX$ consists of those elements $a\in G$ which act on the
set $\seq y1s$ by a permutation preserving the multiplicities of the points.
Let $(\seq{y}{i_1}{i_\ell})$ be a cycle of the permutation defined by the action of $a$ on
$\{\seq y1s\}$. The multiplicities
$\seq{\mu}{i_1}{i_\ell}$ are equal to each other. Let us define
$\beta(\seq{y}{i_1}{i_\ell})$ as $\alpha_{y_{i_1}}(a^{\ell})$. Now
$\alpha^{(k)}_{\{x_i\}}(a)$ is the product
of the factors
$(\beta(\seq{y}{i_1}{i_\ell}))^{\mu_{i_1}}$
over all the cycles
$(\seq{y}{i_1}{i_\ell})$ in the permutation defined by $a$.

\begin{remark}
The reason for this definition is inspired by the application of this notion below
and is the following one. One has to explain it for an irreducible
equipped $G$-set: the $k$-th symmetric power of the disjoint
union of two equipped
$G$-sets
$\w{X}'=(X',\alpha')$ and $\w{X}''=(X'',\alpha'')$ is the disjoint union of the products
$S^i\w{X}'\times S^{k-i}\w{X}''$ over all $i=0, 1, \ldots, k$. Assume that the group $G$
acts on a germ $(V,0)$ of a complex analytic space (and thus on the ring $\OO_{V,0}$
of germs of functions on it and on its projectivization
$\P\OO_{V,0}$) and that an irreducible
equipped $G$-set $\w{X}=(X,\alpha)$ is represented by the orbit $Gh$ of a function germ
$h\in \P\OO_{V,0}$ (or rather of its class there). Let $G_h$ be the isotropy subgroup of
$h$ in $\P\OO_{V,0}$. This means that the $G$-set  $X=Gh\cong
G/G_h$ consists of $\vert G\vert/\vert G_h\vert$
points represented by a set of function germs, for an element $a\in G$ and for a function $h'$
from the set of function germs representing the orbit of $h$, the function $a^*h'$ is, up to a
constant factor, another function-germ from this set, and, for $a\in G_{h'}$,
$a^*h'=\alpha_{h'}(a)h'$. Let us consider all functions of the form $h_1\cdot
h_2\cdot\ldots\cdot h_k$, where $h_1$, $h_2$, \dots $h_k$
are functions from the set representing $Gh$. In the application below two products of this sort
can coincide in $\P\OO_{V,0}$ if and only if they consist of
the same functions (in an arbitrary order,
of course). (This will be the case since all the functions from the orbit will have different
zero-level curves.) In this case this set of functions is, in the natural way, isomorphic to the
$k$-th symmetric power $S^kX$ of the $G$-set $X$.
Let us write a product of the functions representing $Gh$ as
$\{h^*=h_1^{\mu_1}\cdot\ldots\cdot h_s^{\mu_s}\}$ where $\seq h1s$ are different
functions from the
set representing $Gh$ and $\mu_i$ are their multiplicities. The isotropy subgroup of
the function $h^*$ in $\P\OO_{V,0}$ consists of those elements $a\in G$ which act on the
set $\seq h1s$ by a permutation preserving the multiplicities of the functions.
The function $h^*$ is the product of the functions corresponding to cycles of the permutation.
Let $(\seq{h}{i_1}{i_\ell})$ be a cycle of the permutation defined by $a$. The
multiplicities $\seq{\mu}{i_1}{i_\ell}$ of these functions in the product are equal to each other.
The element $a$ acts on the product $\prod\limits_{j=1}^{\ell}h_{i_j}$
by the multiplication by $\alpha_{h_{i_1}}(a^{\ell})$. Thus the element $a$ acts on
the function $h^*$
by the multiplication by the product
of the factors
$(\alpha_{h_{i_1}}(a^{\ell}))^{\mu_{i_1}}$
over all the cycles
$(\seq{y}{i_1}{i_\ell})$ in the permutation defined by $a$.
\end{remark}

From the definition (\ref{lambda})
one can see that
$$
(1-t)^{-[\w{X}\cup \w{Y}]} = (1-t)^{-[\w{X}]} (1-t)^{-[\w{Y}]}\; .
$$
Therefore this definition in the obvious way extends to the Grothendieck ring
$\w{A}(G)$ defining a pre-$\lambda$-structure on it.

\begin{example}
For $G=\Z_2$,
\begin{equation*}
\begin{aligned}
(1-t)^{[\Z_2/\Z_2]_{\sigma}} &= 1 + [\Z_2/\Z_2]_{\sigma} t + t^2 +
[\Z_2/\Z_2]_{\sigma} t^3 + \cdots  = \frac{1 +[\Z_2/\Z_2]_{\sigma}t}{1-t^2}\,,
\\
(1-t)^{[\Z_2/(e)]} & = 1 + [\Z_2/(e)] t + ([\Z_2/(e)]+1)t^2 + \cdots \\
& \qquad \qquad \cdots +
(k [\Z_2/(e)]+1)t^{2k} + (k+1)[\Z_2/(e)]t^{2k+1} +
\cdots \\
& = \frac{1}{1-t^2} + \frac{t}{(1-t)(1-t^2)}[\Z_2/(e)]\;.
\end{aligned}
\end{equation*}
\end{example}

Let $(V,0)$ be a germ of a complex analytic space
with an action of a finite group $G$ and let $\OO_{V,0}$ be the ring of germs of
functions on it. Without loss of generality we assume that the $G$-action on $(V,0)$
is faithful. The group $G$ acts on $\OO_{V,0}$ by
$$
a^* f(z)=f(a^{-1}z)
$$
where $z\in V$ and $a\in G$.

A valuation $\nu$ on the ring $\OO_{V,0}$
is a function $\nu: \OO_{V,0}\to \Z_{\ge 0}\cup \{+\infty\}$ such that:
\begin{enumerate}
\item[1)] $\nu(\lambda f) = \nu (f)$ for $\lambda\in \C^{*}$;
\item[2)] $\nu(f+g)\ge \min \{\nu(f), \nu(g)\}$;
\item[3)] $\nu(fg)= \nu(f)+\nu(g)$.
\end{enumerate}
A function $\nu:\OO_{V,0}\to \Z_{\ge 0}\cup \{+\infty\}$ which
possesses the properties 1)
and 2) is called an {\it order function}. A collection of order
functions $\seq{\nu}1r$ on
$\OO_{V,0}$ defines an $r$-index filtration on $\OO_{V,0}$:
$$
J(\vv) = \{h\in \OO_{V,0} : \nuv(h)\ge \vv\}\; ,
$$
where $\vv=(\seq v1r)\in \Z_{\ge 0}^{r}$,
$\nuv(h) = (\nu_1(h), \ldots, \nu_r(h))$ and
$\vv'=(\seq{v'}1r)\ge \vv'' =(\seq{v''}1r)$ if and only if
$v'_i\ge v''_i$ for all $i$.

The notion of the Poincar\'e series $P_{\{\nu_i\}}(\tt)$ of the filtration
$\{J(\vv)\}$ (or of the collection $\{\nu_i\}$ of order functions)
was introduced in \cite{CDK}. In \cite{IJM} it was shown that it is
equal to a certain integral with respect to the Euler characteristic over the
projectivization $\P\OO_{V,0}$ of the space $\OO_{V,0}$. The coefficient at
$\tt^{\vv}$ ($\tt=(t_1, \ldots, t_r)$, $\vv=(v_1, \ldots, v_r)$,
$\tt^{\vv}=t_1^{v_1}\cdot\ldots\cdot t_r^{v_r}$) in the Poincar\'e series $P_{\{\nu_i\}}(\tt)$
is equal to the Euler characteristic of the set of function germs $h\in \P\OO_{V,0}$
such that $\nunu(h)=\vv$.

One of the problems to define an equivariant version of the Poincar\'e series is
related to the fact
that an order function $\nu$ on $\OO_{V,0}$ is, in general, not invariant with respect to the $G$-action.
Therefore a $G$-orbit of a function does not correspond to a well-defined monomial
of the form $\tt^{\vv}$.
One can restrict oneself to only $G$-equivariant order functions on $\OO_{V,0}$. However this
makes the construction rather poor. Instead of that one can associate to an orbit the sum of values
of the $G$-shifts of the order function, i.e. the sum of the values of the order functions
$a^*\nu$ for $a\in G$. This leads to a meaningful notion. E.g., if $\nu_1$, \dots, $\nu_r$
are the curve valuations defined by irreducible plane curve singularities $(C_i, 0)\subset(\C^2,0)$
(see Section~2) and $\mu(h):=\nu_1(h)+\ldots+\nu_r(h)$ then the integral $\int_{\P\OO_{V,0}}t^{\mu(h)}d\chi$
is equal to $\Delta^C(t, \ldots, t)$, where $\Delta^C$ is the Alexander polynomial (in several
variables) of the link of the curve singularity $C=\bigcup_{i=1}^r C_i$, which in its turn
coincides with the monodromy zeta function of an equation of the curve $C$.

The group $G$ acts on the set of (order) functions on
$\OO_{V,0}$. Let $\seq{\nu}1r$ be a collection of order functions
on $\OO_{V,0}$. Let $\omega_i:\OO_{V,0}\to \Z_{\ge 0}\cup
\{+\infty\}$ be defined by $\omega_i=\sum_{a\in G}a^* \nu_i$. The
functions $\omega_i$ are $G$-invariant.

The functions $\omega_i$ are not, in general, order functions.
Assume that the order functions $\seq{\nu}1r$ are finitely determined. This means
that,  for each $k\in\Z_{\ge 0}$, the set $\{h\in \OO_{V,0}:
\nu_i(h)=k\}$
is cylindric in the sense of \cite{IJM}. For an element $h\in
\P\OO_{V,0}$, that is
for a function germ considered up to a constant factor, let $G_h$ be the isotropy
subgroup $G_h=\{a\in G: a^*h=\alpha_h(a)h\}$
($a\mapsto \alpha_h(a)\in\C^*$ determines
a one-dimensional representation of the subgroup $G_h$) and let
$Gh\cong G/G_h$ be the orbit
of $h$ in $\P\OO_{V,0}$. Let $\widetilde{X}_h$ be the element of
the ring $\widetilde{A}(G)$
represented by the $G$-set $Gh$ with the representation
$\alpha_{a^*h}$ associated to the point $a^*h\in Gh$ ($a\in G$).
The correspondence $h\mapsto \widetilde{X}_h$ defines a function
($\widetilde{X}$) on $\P\OO_{V,0}/G$
with values in $\widetilde{A}(G)$.

The usual (non-equivariant) Poincar\'e series of a collection of order functions $\{\nu_i\}$
is the integral with respect to the Euler characteristic over the projectivization $\P\OO_{V,0}$
of the function $\tt^{\nunu(h)}$ with values in $\Z[[t_1, \ldots, t_r]]$. The equivariant
version will be defined as the integral over $\P\OO_{V,0}/G$ of the function
$\widetilde{X}_h\tt^{\underline{\omega}(h)}$ with values in $\widetilde{A}(G)[[t_1, \ldots, t_r]]$:
see below. However one has to make the notion of integration with respect to the Euler characteristic
over $\P\OO_{V,0}/G$ more precise. The reason is that the function $\widetilde{X}_h\tt^{\underline{\omega}(h)}$
(or rather $\widetilde{X}_h$ itself) is not cylindric: the condition that $G_h=H$ is not determined
by a jet of the germ $h$ of any order. Therefore one has to change the notion of measurable subsets
of $\P\OO_{V,0}/G$ (i.e., of those subsets for which the Euler characteristic is defined) a little bit.

The quotient $\P\OO_{V,0}/G$ is decomposed into the disjoint parts $(\P\OO_{V,0}/G)^{[H], \alpha}$
for all the conjugacy classes $[H]$ of subgroups of $G$ and all the one-dimensional representations
$\alpha$ of $H$, where $(\P\OO_{V,0}/G)^{[H], \alpha}$ consists of the functions $h\in\P\OO_{V,0}/G$
with the isotropy subgroup $G_h$ conjugate to $H$  and, for those of them with $G_h=H$, the corresponding
one-dimensional representation of $H$ is $\alpha$. Let $\overline{(\P\OO_{V,0}/G)^{[H], \alpha}}$
be the set of $h\in\P\OO_{V,0}/G$ such that $H$ is conjugate to a subgroup of $G_h$ and,
for those of them with $H\subset G_h$, the restriction of the corresponding one-dimensional representation
of $G_h$ to $H$ coincides with $\alpha$. (The set $\overline{(\P\OO_{V,0}/G)^{[H], \alpha}}$ is
in some sense the closure of $(\P\OO_{V,0}/G)^{[H], \alpha}$.) In the usual way one can define
measurable (i.e., cylindric) subsets of $\overline{(\P\OO_{V,0}/G)^{[H], \alpha}}$.
Now a subset $A$ of $(\P\OO_{V,0}/G)^{[H], \alpha}$ will be called measurable if it is the
intersection with $(\P\OO_{V,0}/G)^{[H], \alpha}$ of a cylindric subset $B$ in
$\P\OO_{V,0}/G  \left( =\overline{(\P\OO_{V,0}/G)^{[(e)], 1}}\;\right)$.
Its measure
(the Euler characteristic) is defined by the following recurrent equation
$$
\chi(A)=\chi\left(B\cap \overline{(\P\OO_{V,0}/G)^{[H], \alpha}}\right)-
\sum\chi\left(B\cap (\P\OO_{V,0}/G)^{[H'], \alpha'}\right)\,,
$$
where the last sum is over all the conjugacy classes $[H']\in\mbox{Conjsub\,}G$ such that
$H\subset H'$, $H\ne H'$ and $\alpha'_{\vert H}=\alpha$.
(This equation is a recurrent one since it assumes that the measures (the Euler characteristics)
of subsets of $(\P\OO_{V,0}/G)^{[H'], \alpha'}$ with $H'\supset H$, $H'\ne H$ are already defined.

\begin{definition}
 The {\it equivariant Poincar\'e series} $P^G_{\{\nu_i\}}(\tt)$ of the collection $\{\nu_i\}$
 is defined by
 \begin{equation}\label{main_definition}
 P^G_{\{\nu_i\}}(\tt)=\int_{\P\OO_{V,0}/G}
\widetilde{X}_h\tt^{\underline{\omega}(h)} d\chi\in
 \widetilde{A}(G)[[t_1, \ldots, t_r]]\; ,
 \end{equation}
where $\tt^{\underline{\omega}(h)}=t_1^{\omega_1(h)}\cdot\ldots\cdot
t_r^{\omega_r(h)}$, $t_i^{+\infty}$ should be regarded as 0.
\end{definition}

In other words
$$
P^G_{\{\nu_i\}}(\tt)=
\sum
\chi(\{h\in \P\OO_{V,0}: \underline{\omega}(h) = \ww, \widetilde{X}_h=[G/H]_{\alpha}\}/G)
[G/H]_{\alpha}\tt^{\ww}\;,
$$
where the sum is over all the conjugacy classes $[H]$ of subgroups in $G$, all the
one-dimensional representations $\alpha$ of $H$ and $\ww\in\Z_{\ge0}^r$.

Applying the reduction homomorphism $\rho:\widetilde{A}(G)\to A(G)$ to the Poincar\'e
series $P^G_{\{\nu_i\}}(\tt)$, i.e. to its coefficients, one gets the series
$\rho P^G_{\{\nu_i\}}(\tt)\in A(G)[[t_1, \ldots, t_r]]$, i.e. a power series
with the coefficients from the (usual) Burnside ring.
Applying the homomorphism $\widehat{\rho}:\widetilde{A}(G)\to \Z$ and
$\eps:\widetilde{A}(G)\to R_1(G)$,
one gets the series $\widehat{\rho} P^G_{\{\nu_i\}}(\tt)\in \Z[[t_1, \ldots, t_r]]$
and  $\eps P^G_{\{\nu_i\}}(\tt)\in R_1(G)[[t_1, \ldots, t_r]]$
respectively.

\begin{statement}
 One has
$$
\widehat{\rho} P^G_{\{\nu_i\}}(\tt)=P_{\{a*\nu_i\}}(t_1, \ldots, t_1, t_2, \ldots,
t_2,
\ldots, t_r, \ldots, t_r)\,,
 $$
 where $P_{\{a*\nu_i\}}(\cdot)$ is the usual (non-equivariant) Poincar\'e series of
the collection
 of $\vert G\vert r$ order functions $\{a*\nu_1, a*\nu_2,\ldots, a*\nu_r\vert a\in
G\}$
(each group
 of equal variables in $P_{\{a*\nu_i\}}$ consists of $\vert G\vert$ of them).
One has
 $$
 \eps P^G_{\{\nu_i\}}(\tt)=P^G(t_1^{\vert G\vert}, t_2^{\vert G\vert}, \ldots, t_r^{\vert G\vert})\,,
 $$
 where $P^G(\tt)$ is the equivariant Poincar\'e series defined in \cite{MMJ} (actually only
 for divisorial and curve valuations on $\OO_{\C^2,0}$) as an element of
the power series ring
 $R_1(G)[[t_1, t_2, \ldots, t_r]]$.
\end{statement}

In \cite{RMC} an equivariant version of the Poincar\'e series was defined not as a power series,
but as an element of a rather big (and sophisticated) Grothendieck ring of so-called
locally finite
$(G,r)$-sets ($G$-sets with an additional structure).

A {\it locally finite} $(G,r)$-{\it set} is a triple $(X, \vv, \alpha)$ where
\begin{itemize}
\item $X$ is a $G$-set;
\item $\vv$ is a function on $X$ with values in
$\Z^r_{\ge 0}$;
\item $\alpha$ associates to each point $x\in X$ a one-dimensional representation $\alpha_{x}$ of the isotropy
subgroup $G_x= \{ a\in G : ax =x\}$ of the point $x$;
\end{itemize}
satisfying the following conditions:
\begin{enumerate}
\item[1)] $\alpha_{ax}(b)=\alpha_x (a^{-1}ba)$ for $x\in X$, $a\in G$, $b\in G_{ax}=a G_x a^{-1}$;
\item[2)] for any $\vv\in \Z^r_{\ge 0}$ the set $\{x\in X:
\vv(x)\le \vv\}$ is finite.
\end{enumerate}

The equivariant version of the Poincar\'e series in \cite{RMC} is a virtual locally finite $(G,r)$-set
$P^G_{\{\nu_i\}}=(X, \vv, \alpha)$. Consider the function
$\underline{\omega}$ on $X$ with values in $\Z^r_{\ge 0}$ defined
by
$$
\underline{\omega}(x)=\sum_{a\in G}\vv(ax)\,.
$$
This function is $G$-invariant and thus is constant on irreducible components of the $(G,r)$-set.
For $\ww\in\Z_{\ge 0}^r$, let $X_{\ww}=\{x\in X:
\underline{\omega}(x)=\ww\}$. The $G$-set $X_{\ww}$ with the
representations $\alpha_x$ associated to its points is a finite equipped $G$-set
$\widetilde{X}_{\ww}=(X_{\ww},\alpha_{\vert X_{\ww}})$. One can see that the equivariant
Poincar\'e series (\ref{main_definition}) is equal to
$$
P^G_{\{\nu_i\}}(\tt)=\sum_{\ww\in\Z_{\ge
0}^r}[(X_{\ww},\alpha_{\vert X_{\ww}})]\tt^{\ww}\,.
$$
Thus the equivariant version of the Poincar\'e series from \cite{RMC} determines the
Poincar\'e series considered here.

\section{The equivariant Poincar\'e series for curve
and divisorial valuations on $\OO_{\C^2,0}$}\label{sec2}
Here we write an A'Campo type formula for the equivariant Poincar\'e series mentioned in the
title of the section in terms of a $G$-resolution. We shall treat two types of plane valuations.

Let $(C,0)\subset(\C^2,0)$ be an irreducible germ of a plane curve and let
$\varphi: (\C,0)\to(\C^2,0)$
be a parametrization (uniformization) of it (i.e. $\mbox{Im\,}\varphi=(C,0)$ and
$\varphi$ is an isomorphism between $(\C,0)$ and $(C,0)$ outside of the origin). For
$h\in\OO_{\C^2,0}$, let
$h(\varphi(\tau))= c \tau^{\nu(h)}+\mbox{\,terms of higher
degree}$, where $c\ne0$, $\nu(h)\in\Z_{\ge0}$.
(If $h(\varphi(\tau))\equiv 0$, one defines $\nu(h)$ as $+\infty$.) The function $\nu$ on
$\OO_{\C^2,0}$ is a valuation: the so-called curve valuation.

Let $\pi: ({\cal X, D})\to (\C^2,0)$ be a modification of the plane $(\C^2,0)$
by a sequence of blowing-ups. Its exceptional divisor ${\cal D}=\pi^{-1}(0)$ is
the union of irreducible components $E_{\sigma}$, $\sigma\in\Gamma$, each of them
is isomorphic to the complex projective line $\C\P^1$. For a component $E_{\sigma}$
of the exceptional divisor, and for $h\in \OO_{\C^2,0}$, let
$\nu_\sigma(h)$ be the order of zero of the lifting $h\circ\pi$ of the function $h$
to the space ${\cal X}$ of the modification along the component $E_\sigma$.
The function $\nu_\sigma$ on $\OO_{\C^2,0}$ is a valuation: the so-called
divisorial valuation (corresponding to the divisor $E_{\sigma}$).

Assume that a finite group $G$ acts on $(\C^2,0)$ (by a representation).
Let $\nu_i$, $i=1, \ldots, r$,
be either a curve or a divisorial valuation on $\OO_{\C^2,0}$. We shall write
$I_0=\{1,2,\ldots,r\}=I'\sqcup I''$, where $i\in I'$ if and only if the corresponding valuation
$\nu_i$ is a curve one. For $i\in I'$, let $(C_i,0)$ be the
plane curve defining the valuation $\nu_i$.

A $G$-{\it equivariant resolution} (or a $G$-{\it resolution} for short) of
the collection $\{\nu_i\}$ of valuations is a proper complex
analytic map $\pi: ({\cal X, D})\to (\C^2,0)$ from a
smooth surface ${\cal X}$ with a $G$-action such that:
\begin{enumerate}
\item[1)]
$\pi$ is an isomorphism outside of the origin in $\C^2$;
\item[2)]
$\pi$ commutes with the $G$-actions on ${\cal X}$ and on $\C^2$;
\item[3)]
the total transform $\pi^{-1}(\bigcup\limits_{i\in I',\, a\in G}
aC_i)$ of the curve $GC=G(\bigcup\limits_{i\in I'} C_i) $
is a normal crossing divisor on ${\cal X}$ (in particular,
the exceptional
divisor ${\cal D}=\pi^{-1}(0)$ is a normal crossing divisor as
well);
\item[4)]
for each branch $C_i$, $i\in I'$, its strict transform
$\widetilde{C}_i$ is a germ of a smooth curve transversal to the
exceptional divisor ${\cal D}$ at a smooth point $x$ of
it and is invariant with respect to the isotropy subgroup
$G_x=\{g\in G: gx=x\}$ of the point $x$;
\item[5)] for each $i\in I''$, the exceptional divisor ${\cal D}=\pi^{-1}(0)$
contains the divisor defining the divisorial valuation $\nu_i$.
\end{enumerate}

A $G$-resolution can be obtained by a $G$-invariant sequence of blow-ups of points.
The condition 4) means, in particular, that $\pi$ is an embedded resolution of the curve
$GC=\bigcup_{i\in I',\, a\in G} aC_i$.

Let ${\stackrel{\circ}{\cal D}}$ be the ``smooth part" of
the exceptional divisor ${\cal D}$ in the total transform
$\pi^{-1}(GC)$ of the curve $GC$, i.e., ${\cal D}$ itself minus
all the intersection points of its components
and all the intersection points with the components of the strict transform of the curve $GC$.
For $x\in {\stackrel{\circ}{\cal D}}$, let $\w{L}_x$ be a germ
of a smooth curve on ${\cal X}$ transversal to
${\stackrel{\circ}{\cal D}}$ at the point $x$ and invariant
with respect to the isotropy subgroup $G_x$ of the point $x$.
The image $L_x=\pi(\w{L}_x)\subset (\C^2,0)$ is called a {\it
curvette} at the point $x$. Let the curvette $L_x$ be given by an
equation $h_x=0$, $h_x\in\OO_{\C^2,0}$. Without loss of generality
one can assume that the function germ $h_x$ is $G$-equivariant.
Moreover we shall assume that the germs $h_x$ associated to different
points $x\in {\stackrel{\circ}{\cal D}}$ are choosen so that
$h_{ax}(a^{-1}z)/h_x(z)$ is a constant (depending on $a$ and $x$).
Also we shall assume that the dependence of the germ $h_x$ on the point $x$
is constructible, i.e. depends on $x$ analytically on $x\in {\stackrel{\circ}{\cal D}}$
except at a finite set of points.

Let $E_\sigma$, $\sigma\in\Gamma$, be the set of all irreducible components of the
exceptional divisor $\cal {D}$ ($\Gamma$ is a $G$-set itself). For $\sigma$ and $\delta$
from $\Gamma$, let $m_{\sigma\delta}:=\nu_{\sigma}(h_x)$, where $\nu_{\sigma}$ is the
corresponding
divisorial valuation, $h_x$ is the germ defining the curvette at a point
$x\in E_{\delta}\cap {\stackrel{\circ}{\cal D}}$. One can show that the matrix $(m_{\sigma\delta})$
is minus the inverse matrix to the intersection matrix $(E_{\sigma}\circ E_{\delta})$
of the irreducible components of the exceptional divisor $\cal {D}$. For $i=1,\ldots,r$,
let $m_{\sigma i}:=m_{\sigma\delta}$, where $\delta$ is the number of the component of $\cal {D}$
corresponding to the valuation $\nu_i$, i.e. either the component defining the valuation
$\nu_i$ if $\nu_i$ is a divisorial valuation (i.e. if $i\in I''$), or the component intersecting
the strict transform of the corresponding irreducible curve $C_i$ if $\nu_i$ is a curve valuation
(i.e. if $i\in I'$). Let $\mm_{\sigma}:=(m_{\sigma 1},\ldots,
m_{\sigma r})\in \Z_{\ge 0}^r$.

Let $\{\Xi\}$ be a stratification of the smooth curve
$\widehat{\cal D}={\stackrel{\circ}{\cal D}}/G$ such that:
\begin{enumerate}
\item[1)] each stratum $\Xi$ is connected;
\item[2)] for each point $\h{x}\in \Xi$ and for each point $x$
from its pre-image $p^{-1}(\h{x})$, the conjugacy class of the
isotropy subgroup $G_x$ of the point $x$ is the same, i.e., depends only on
the stratum $\Xi$.
\end{enumerate}
The latter is equivalent to say that the factorization map
$p: {\stackrel{\circ}{\cal D}} \to \h{\cal D}$ is a covering
over each stratum $\Xi$.

For a point $x\in {\stackrel{\circ}{\cal D}}$, let
$\widetilde{X}_x$ be the
equipped $G$-set defined by $\widetilde{X}_x=\widetilde{X}_{h_x}$, where $h_x$
is the $G_x$-equivariant function defining the choosen curvette at the point $x$ (see above).
The equipped $G$-set $\widetilde{X}_x$ is one and the same for all points $x$ from
the preimage of a stratum $\Xi$ and therefore it defines an element
$[\widetilde{X}_{\Xi}]\in \widetilde{A}(G)$.
Let
\begin{equation}\label{2}
\underline{\omega}_x:=\underline{\nu}(\prod_{a\in G}h_{ax})\,.
\end{equation}
One can see that, for a stratum $\Xi$ of the stratification, $\underline{\omega}_x$
is one and
the same for all points $x$ from the preimage $\pi^{-1}(\Xi)$. We shall denote it by
$\underline{\omega}_{\Xi}$.

\begin{theorem}
 For a collection $\{\nu_i\}$ of curve and divisorial valuations one has
 \begin{equation}\label{ACampo}
 P^G_{\{\nu_i\}} (\tt) =\prod_{\Xi}
(1-\tt^{\underline{\omega}_{\Xi}})^{-\chi(\Xi)[\w{X}_{\Xi}]}\,.
 \end{equation}
\end{theorem}

\begin{proof}
Let $Y$ be the configuration space of effective divisors on
$\stackrel{\circ}{\cal D}$. The space $Y$ can be represented as
$$
Y = \bigsqcup_{\{k_{\alpha}\}} \left( \prod_{\sigma\in
\Gamma}S^{k_{\alpha}} \stackrel{\circ}{E}_{\sigma} \right) =
\prod_{\sigma\in \Gamma}\left( \bigsqcup_{k=0}^{\infty}S^k
\stackrel{\circ}{E}_{\sigma}
\right)\; ,
$$
where
$\oE_{\sigma}= E_{\sigma}\cap \oD$,
 $S^k Z = Z^k/S^k$ is the $k$-th symmetric power of the
space $Z$. One has the natural $G$-action on the space $Y$. Let
$\widehat{Y}=Y/G$ be the space of $G$-orbits in $Y$.

For a point $y\in Y$, $y = \sum_{i=1}^{n}x_i$, let $\w X_y$ be
the element of the ring $\w{A}(G)$ defined by
$\w X_y=\w X_{\prod_{i}h_{x_i}}$. This way one gets a
$G$-invariant map
$\w X:Y\to\w{A}(G)$ and thus a map
$\h{X}: \h{Y}\to \w{A}(G)$.

Let $\bb{Y}$ be the configuration space of effective
divisors on
the smooth curve $\h{\cal D}=\stackrel{\circ}{\cal D}/G$. The space $\bb{Y}$
can be represented as
$$
\bb Y = \bigsqcup_{\{k_{\Xi}\}} \left( \prod_{\Xi
}S^{k_{\Xi}} \Xi \right) =
\prod_{\Xi}\left( \bigsqcup_{k=0}^{\infty}S^k
\Xi
\right)\; .
$$

Let $\bb X: \bb Y \to \w A(G)$ be the function on $\bb Y$ defined in
the following way. Let $\bb y = \sum_{i=1}^{s} \ell_{i}\h{x}_i$,
where $\h{x}_i$ are different points from
$\h{\cal D}=\stackrel{\circ}{\cal D}/G$. Then
$$
\bb{X}(\bb y) := \prod_{i=1}^{s} [S^{\ell_i} \w X_{h_{x_i}}] \; ,
$$
where $x_i$ is a point from the preimage $p^{-1}(\h{x}_i)$ of
$\h{x_i}$.

There is a natural map $q: Y\to \bb Y$ which sends a point
$y=\sum_{i=1}^{n} x_i\in Y$ to the point
$\bb y = \sum_{i=1}^{n} \h{x}_i$, where $\h{x}_i=p(x_i)$ is the
orbit of the point $x_i$.
Let $\underline{\omega}_{\bb y} := \sum_{i=1}^{n}
\underline{\omega}_{x_i}$, where
$\underline{\omega}_{x}$ is $\underline{\omega}(h_x)$ (note
that
$\underline{\omega}_{x}$ depends
only on the component of $\stackrel{\circ}{\cal D}$ containing $x$).

The preimage of a point $\bb y\in \bb Y$ in $Y$ can be described in
the following way. Let
$\bb{y} = \sum_{i=1}^{s}\ell_i \h{x}_i$, where $\h{x}_i$ are
different points from $\h{\cal D}= \stackrel{\circ}{\cal D}/G$. Then
$q^{-1}(\bb y)= \prod_{i=1}^{s} S^{\ell_i}(p^{-1}(\h{x}_i))$.
The definition of the one-dimensional representation associated to points of the
symmetric powers of an equipped $G$-set (or rather the explanation of its meaning
after the definition) shows that as an equipped $G$-set
$q^{-1}(\bb y)$ is isomorphic to $\bb X(\bb y)$

Thus one has:
\begin{equation*}
\begin{aligned}
\ & \int\limits_{\bb Y} \bb{X}(\bb y) \tt^{\underline{\omega}_{\bb
y}} d\chi  =
\prod_{\Xi} \left\{
\sum_{k=0}^{\infty} \left(\sum_{\{k_i\}: \sum i k_i =k} \chi\left(
\frac{\Xi^{{\scriptstyle \sum k_i}} \setminus \Delta}{{
\prod\limits_{i}}S_{k_i}}\right)
{\textstyle \prod\limits_{i}} [S^i X_{\Xi}]^{k_i} \right)
\tt^{k \underline{\omega}_{\Xi}}
\right\} \\
&=
\prod_{\Xi} \left\{
\sum_{k=0}^{\infty} \left( \sum_{\{k_i\}
}
\frac{\chi(\Xi)(\chi(\Xi)-1)\cdots (\chi(\Xi)- {\scriptstyle \sum}
k_i +1)}{\prod\limits_i (k_i!)}
{\textstyle \prod\limits_{i}} [S^i X_{\Xi}]^{k_i}
\right)
\tt^{k \underline{\omega}_{\Xi}}
\right\}
\\
& \qquad \qquad = \prod_{\Xi}\left(
1 + [X_{\Xi}]\; \tt^{\underline{\omega}_{\Xi}} + [S^2X_{\Xi}]\;
\tt^{2
\underline{\omega}_{\Xi}}
+
\cdots
\right)^{\chi(\Xi)} \;,
\end{aligned}
\end{equation*}
where $\Delta$ is the big diagonal in $\Xi^{{\scriptstyle \sum k_i}}$, i.e.
the set of points $(x_1, \ldots, x_{\sum k_i})\in \Xi^{{\scriptstyle \sum k_i}}$
with at least two coinciding components $x_j$.
The last equation follows from the well-known equation
\begin{equation*}
\begin{aligned}
1 +
\sum_{k=0}^{\infty}  \left( \sum_{\{k_i\}:  \sum i k_i =k}
\frac{m(m-1)(m-2)\cdots (m-\sum
k_i +1)}{\prod\limits_i (k_i!)}
{\textstyle \prod\limits_{i}} a_i^{k_i}
\right) t ^k  & = \\
= (1 + a_1 t+ a_2t^2 \cdots )^{m}\,, & \qquad \qquad \\
\end{aligned}
\end{equation*}
see, e.g., \cite{stanley}.

According to the described pre-$\lambda$-structure on the ring $\w{A}(G)$ one
has:
\begin{equation*}
\int_{\bb{Y}} \bb{X}(\bb{y}) \tt^{\underline{\omega}_{\bb{y}}} d\chi
=
\prod_{\Xi}
\left( 1-\tt^{\underline{\omega}_{\Xi}}
\right)^{-\chi(\Xi) [X_{\Xi}]}\; .
\end{equation*}

It is sufficient to prove the equation (\ref{ACampo}) modulo the terms
of degree greater than
$\WW$ for an arbitrary $\WW\in \Z_{\ge 0}^r$. We can assume that the resolution
$\pi:(\XX,\DD)\to (\C^2,0)$ is such that for any $h\in
\OO_{\C^2,0}$ with
$\ww(h)\le\WW$,
the strict transform of the curve $\{h=0\}$ intersects the
exceptional divisor
$\DD$ only at points of $\oD$. Such resolution can be obtained from
an arbitrary one by
additional $G$-invariant blowing-ups of intersection points of components of the
total transform $\pi^{-1}(GC)$ of the curve $GC$. These blowing-ups add, to the
stratification $\{\Xi\}$ of $\h{\cal D}$, strata with zero Euler characteristic and
thus do not change the right hand side of (\ref{ACampo}). Let
$\OO^{\WW}_{\C^2,0}= \{h\in \OO_{\C^2,0} : \ww(h)\le \WW
\}\; .
$
Let $I:\P\OO^{\WW}_{\C^2,0}\to Y$ be the map which sends a
function
$h\in \OO^{\WW}_{\C^2,0}$ to the set of intersection points of
the strict transform
of the zero-level curve $\{h=0\}$ with the exceptional divisor
$\DD$ counted with the
corresponding multiplicities. One has the commutative diagram:
$$
\begin{CD}
\P\OO^{\WW}_{\C^2,0} @>I>> Y \\
@VpVV @VVpV \\
\P\OO^{\WW}_{\C^2,0}/G @>\h{I}>> \h{Y}
\end{CD}
$$
The preimages $I^{-1}(y)$ of points from $Y$ are affine subspaces of
$\P\OO^{\WW}_{\C^2,0}$ (see Proposition 2 in \cite{IJM}) and thus have Euler
characteristic equal to $1$. This implies that the Euler characteristics of the
preimages with respect to $\h{I}$ of points of $\h{Y}$ in $\P\OO^{\WW}_{\C^2,0}/G$
are also equal to $1$.

One has $\underline{\omega}=\h{\underline{\omega}}\circ \h{I}$,
but, in general,
$\w X\neq
\h{X}\circ \h{I}$ \,because the
isotropy subgroup of $h\in \P\OO^{\WW}_{\C^2,0}$ can be different
from the isotropy subgroup of
its image in $Y$ (being a proper subgroup of the latter).
Therefore the $G$-orbits of a point in
$\P\OO^{\WW}_{\C^2,0}$
and of its image in $Y$
can be different as $G$-sets.

To compute the integral of the function
$\w X_{h}\tt^{\underline{\omega}(h)}$ over
$\P\OO^{\WW}_{\C^2,0}/G$, we shall, for each point $\h{y}\in \h{Y}$, construct a
point
$h_{\h{y}}$ in $\h{I}^{-1}(\h{y})$ so that $\w X_{h_{\h{y}}}= \w{X}(\w y) =
\h{X}(\h{y})$ and the complement
$\h{I}^{-1}(\h{y})\setminus \{h_{\h{y}}\}$ is fibred into $\C^*$ families of points
with the function $\w X_{h}$ constant along the fibres. This
implies that
$$
\int_{\h{I}^{-1}(h)} {\widetilde{X}}_h d \chi = \h{X}(\h{y})
$$
and the Fubini formula applied to the map $\h{I}$ gives (up to terms of degree
greater than $\WW$)
\begin{equation*}
\int\limits_{\P\OO_{\C^2,0}/G} \w
X_{h}\tt^{\underline{\omega}(h)} d \chi =
\int\limits_{\h{Y}} \h X(\h y)\tt^{\underline{\omega}_{\h y}} d \chi =
\int\limits_{\bb{Y}} \bb X(\bb y)\tt^{\underline{\omega}_{\bb y}} d \chi =
\prod_{\Xi} (1-\tt^{\ww_{\Xi}})^{-\chi(\Xi) [X_{\Xi}]}\; .
\end{equation*}

Let  $\widehat{y}\in\widehat{Y}$ be the orbit of
$y=\sum\limits_{i=1}^{m} x_i \in Y$ and let
$h_y :=\prod\limits_{i=1}^{m} h_{x_i}$. The isotropy
subgroup of $h_y$ in $\P\OO_{\C^2,0}$ coincides with the isotropy
subgroup of $y$ in $Y$ and therefore
$\h X_{h_{\widehat{y}}}={\widehat{X}}(\widehat{y})$
(here $h_{\widehat{y}}=p(h_{y})$ is the image of  $h_y$ in $\P\OO_{\C^2,0}/G$).

 Let $h\in I^{-1}(y)$.  The strict transforms of the curves
$\{h=0\}$ and  $\{h_y=0\}$ intersect the exceptional divisor
${\cal D}$ at the same points
with the same multiplicities. Therefore the ratio
$\psi=\dfrac{h\circ \pi}{h_y\circ \pi}$ of the liftings ${h\circ
\pi}$ and ${h_y\circ \pi}$ of the functions
$h$ and $h_y$ to the space ${\cal X}$ of the modification has
neither
zeros no poles on the exceptional divisor ${\cal D}$ and thus is
constant on it.
Therefore (multiplying $h$ by a constant) one can assume that
the ratio $\psi$ is equal to $1$ on ${\cal D}$.

Let $h_{\lambda}:= h_y+ \lambda(h-h_y)$, $\lambda\in\C^*$. One
has $\dfrac{h_{\lambda}\circ \pi}{h_y\circ \pi}=1$ on the
exceptional divisor ${\cal D}$.
Therefore $I(h_{\lambda})=I(h_{y})=y$. Moreover the isotropy
subgroup of each $h_{\lambda}$ in  $\P\OO_{\C^2,0}$ coincides with
the isotropy subgroup of $h$
and therefore $\w X_{h_{\lambda}}$ is constant. This proves the
statement.
\end{proof}

\medskip
A difference between the equation (\ref{ACampo}) and the
equations in Theorems 1  and 2
in \cite{RMC} convinced us that the equations  in \cite{RMC} contain a
mistake. Here we correct this mistake and formulate a somewhat more general statement
than Theorems 1 and 2 in \cite{RMC}. There we considered two cases: a collection of divisorial valuations and a collection
of curve valuations. Here we assume that the collection $\{\seq{\nu}1r\}$ may contain
both kinds of valuations on $\OO_{\C^2,0}$.

Let
$T=(X,\vv,\alpha)$ be  a locally finite $(G,r)$-set.
Let us define its symmetric power $S^kT$
in the following way. The $(G,r)$-set $S^k T$ is the triple
$(S^k X, \vv^{(k)}, \alpha^{(k)})$ where $S^k X$ is the $k$-th
symmetric power of the $G$-set
$X$ (with the natural $G$-action),
$\vv^{(k)}(\{\seq x1k\})= \sum_{i=1}^k \vv(x_i)$, where the
unordered collection
$\{\seq{x}1k\}$ represents a point in $S^kX$ and $\alpha^{(k)}$
is defined as above (for
symmetric powers of equipped $G$-sets). If the locally-finite $(G,r)$-set $T$ is such
that $\vv^{-1}(0)=1$ (i.e. a point), then
$ST = \bigsqcup\limits_{k=0}^{\infty} S^k T$ is a locally finite
$(G,r)$-set as well.
The $(G,r)$-set $T_{\Xi}$ defined in \cite{RMC} has this property. Thus
$ST_{\Xi}$ is defined as an element of
$K_{0}((G,r)-\mbox{sets})$.

\begin{theorem}
One has
$$
P^G_{\{\nu_i\}} = \prod_{\Xi} (ST_{\Xi} )^{\chi(\Xi)}\; .
$$
\end{theorem}

The proof is essentially the one given above.

Theorems 1 and 2 from \cite{RMC} where used in \cite{Arkiv}. Thus from formal point
of view one can assume that the results in \cite{Arkiv} are not fully proved.
However
essentially we did not use the equations of Theorems 1 and 2 but only the fact that
the knowledge of the Poincar\'e series
$P^{G}_{\{\nu_i\}}\in K_0((G,r)-\mbox{sets})$
permitted to restore the numbers
$\sum\limits_{\Xi: T_{\Xi}=T} \chi(\Xi)$. In that paper this followed from the unique
factorization of the Poincar\'e series of the form
$$
P^{G}_{\{\nu_i\}} = \prod_{T}(1-T)^{-s_T}
$$
where the product is over all irreducible elements of the ring
$K_0((G,r)-\mbox{sets})$  such that $\vv(x)>0$ for all $x\in X$, $s_{T}\in\Z$.
However, in the
same way, one has the unique factorization of the form
$$
P^{G}_{\{\nu_i\}} = \prod_{T}(ST)^{s_T}
$$
where $ST=1+T+S^2T + \cdots $ and the product is over all the
irreducible elements of
$K_0((G,r)-\mbox{sets})$ with the mentioned property. This
follows from the fact
that the irreducible elements $T=(X,\vv,\alpha)$ of $K_0((G,r)-\mbox{sets})$ with
$\vv(x)>0$ can be partially ordered in such a way that
\begin{enumerate}
\item[1)] for $k > 1$,  $S^kT$ contains only irreducible
components greater than $T$;
\item[2)] each set of irreducible components has a minimal one.
\end{enumerate}
In this case if $P^{G}_{\{\nu_i\}}=1+  s T_1 +\cdots $ where $T_1$
is a minimal irreducible $(G,r)$-set in the decomposition of
$P^{G}_{\{\nu_i\}}$ with $s\neq 0$ then $P^{G}_{\{\nu_i\}}=(ST_1)^s \cdot Q$
where $Q=1+ \ldots $ (monomials not smaller than $T_1$). Then one can apply  the
factorization procedure to $Q$ and get the result. The required order can be defined,
e.g., in the following way.
For an irreducible $(G,r)$-set $T=(X,\vv,\alpha)$, let
$\ww(T):= \sum_{a\in G} \vv(a x)$ for a point $x\in X$
(one can see that $\ww(T)$ does not depend on $x$). Then we say that
$T_1<T_2$ if and only if $\ww(T_1)<\ww(T_2)$. One can see that a symmetric power of
an irreducible element $T$ (with $\ww(x)>0$) contains only components greater than
$T$.

\section{Relations with the equivariant monodromy zeta functions}\label{sec3}

In \cite{GLM}, \cite{GLM-zeta2}, there were defined two versions of the equivariant
monodromy zeta function for a $G$-invariant function germ $f$ as a series from
$(A(G)\otimes\Q) [[t]]$ (in \cite{GLM}) and from $A(G)[[t]]$ (in \cite{GLM-zeta2}).
We
shall show that in the case when $(V,0)=(\C^2,0)$ (i.e. on the plane), with some
exceptions, these monodromy
zeta functions can be restored (using a simple algorithm) from the
equivariant Poincar\'e series
$P^{G}_{\{\nu_i\}}(\tt)\in \w{A}(G)[[\seq t1r]]$
for the set $\{\nu_i\}$ of the curve valuations defined by the components of the
$G$-invariant curve $\{f=0\}$.

Let $(\C^2,0)$ be endowed with a $G$-representation and let $f:(\C^2,0)\to(\C,0)$
be a $G$-invariant function germ. Let $C=\{f=0\}$ be its zero-level curve.
Let $C=\bigcup_{i=0}^{r}C_i$ be the representation of $C$ as the union of
irreducible components (among the curves $C_i$ one can have identical ones), and let
$\nu_i$ be the curve valuation defined by the irreducible curve $C_i$. Let
$P^{G}_{\{\nu_i\}}(\seq t1r)$ be the equivariant Poincar\'e series
of the collection $\{\nu_i\}$ of valuations.

\begin{remark}
Assume that one takes one irreducible component from each orbit of the component of
the curve $C$ with the reduced structure, say, $\seq{C'}1s$.
The equivariant Poincar\'e series $P^{G}_{\{\nu_i\}}(\seq t1r)$ can be obtained from
the equivariant Poincar\'e series $P^{G}_{\{\nu'_i\}}(\seq{t'}1s)$ by substituting
each variable $t'_j$ by the product of the corresponding variables $t_i$.
\end{remark}

Let $\zeta^G_{f}(t)\in (A(G)\otimes\Q)[[t]]$ and
$\w{\zeta}^G_{f}(t)\in A(G)[[t]]$ be the monodromy zeta functions of the
$G$-invariant germ $f$ defined in \cite{GLM} and \cite{GLM-zeta2} respectively. One
cannot hope to restore the Poincar\'e series $P^{G}_{\{\nu_i\}}(\seq{t}1r)$ from the
series $\zeta^G_f(t)$ or $\w{\zeta}^g_f(t)$ since, in particular, the Poincar\'e
series  is a series in a number of variables and thus is a more fine invariant than
the zeta functions. In particular, the monodromy zeta function does not determine the
Poincar\'e series of a non-irreducible plane curve singularity in the non-equivariant
situation, i.e. $G=(e)$. (In this case the Poincar\'e series coincides with the
multi-variable Alexander polynomial which can be considered as a multi-variable
generalization of the monodromy zeta function. The multi-variable Alexander
polynomial, but not the monodromy zeta function, determines the topology of a curve
singularity: \cite{Ya}, \cite{Wall}). Therefore we shall discuss the
possibility to restore the equivariant monodromy zeta functions $\zeta^G_f(t)$ and
$\w{\zeta}^G_f(t)$ from the equivariant Poincar\'e series
$P^G_{\{\nu_i\}}(\seq t1r)$ or from its reduction
$\rho P^G_{\{\nu_i\}}(\seq t1r)\in A(G)[[\seq t1r]]$.

One can easily see that
\begin{equation}
\rho P^{G}_{\{\nu_i\}}(\seq t1r) =
\prod_{\Xi} (1-\tt^{\underline{\omega}_{\Xi}})^{-\chi(\Xi) [G/H_{\Xi}]}
\end{equation}
where $H_{\Xi}=G_x$ for $x\in p^{-1}(\Xi)$
(since $\rho [\w{X}_{\Xi}]=[G/H_{\Xi}]\in A(G)$).

The equivariant zeta functions can be given in terms of a $G$-resolution
$\pi: (\XX,\DD)\to (\C^2,0)$ of the curve $C$. Let $\{\Xi\}$ be the stratification of
the space $\h{\DD}=\oD/G$ described in Section \ref{sec2}.
For $x\in \oD$, let
\begin{equation}\label{**}
n_x = \sum_{i=1}^r \nu_i(h_x)=\abs{\mm_x}\; .
\end{equation}
The isotropy
subgroup
$G_x$ of the point $x$ acts on the $G_x$-invariant normal slice to $\oD$ at the point
$x$ (in fact the strict transform of the $G_x$-invariant curvette $L_x$
corresponding to $x$). Let $\h{G}_x\subset G_x$ be the isotropy subgroup of a point
of this slice different from $x$.

The integer $n_x$ and the conjugancy class of the pair $(G_x,\h{G}_x)$ are the same
for all the points $x$ from the preimage of a stratum $\Xi$ with respect to the
factorization map $p:\oD\to\oD/G$. Therefore let us denote them by $n_{\Xi}$ and
$(H_{\Xi}, \h{H}_{\Xi})$ respectively.

Then one has (\cite{GLM}, \cite{GLM-zeta2}):
\begin{equation*}
\begin{aligned}
\zeta^G_f (t) &= \prod_{\Xi}
\left(
1-t^{n_{\Xi}}
\right)^{-\frac{\abs{\h{H}_{\Xi}}\chi(\Xi)}{\abs{H_{\Xi}}} [G/\h{H}_{\Xi}]}
\; , \\
\w{\zeta}^G_f (t) &= \prod_{\Xi}
\left( 1-t^{n_{\Xi}\frac{\abs{\h{H}_{\Xi}}}{\abs{H_{\Xi}}}
}\right)^{-\chi(\Xi)[G/\h{H}_{\Xi}]}\; .
\end{aligned}
\end{equation*}

Let us first assume that the action of $G$ on $\C^2\setminus \{0\}$ is free. In this
case $\h{H}_{\Xi}=(e)$ for any $\Xi$. Therefore one has
\begin{equation*}
\begin{aligned}
\zeta^G_f (t) &= \prod_{\Xi}
\left(1-t^{n_{\Xi}}\right)^{-\frac{\chi(\Xi)}{\abs{H_{\Xi}}}
[G/(e)]}
\; , \\
\w{\zeta}^G_f (t) &= \prod_{\Xi}
\left(1-t^{\frac{n_{\Xi}}{\abs{H_{\Xi}}}}\right)^{-\chi(\Xi)[G/(e)]}\; .
\end{aligned}
\end{equation*}
The equations (\ref{2}) and (\ref{**}) imply that
$n_{\Xi}=\frac{\abs{\ww_{\Xi}}}{\abs{G}}$, Therefore the equivariant monodromy zeta
functions $\zeta^G_f(t)$ and
$\w{\zeta}^G_f(t)$ can be restored from the A'Campo type decomposition of the series
$\rho P^{G}_{\{\nu_i\}}(\seq{t}{}{})$ in the obvious way. Thus we have the following
statement.

\begin{proposition}
If the $G$-action on $\C^2\setminus \{0\}$ is free, the reduced equivariant
Poincar\'e series
$\rho P^{G}_{\{\nu_i\}}(\seq{t}{1}{r})$ determines the equivariant zeta functions
$\zeta^G_f(f)$ and
$\w\zeta^G_f(f)$.
\end{proposition}

\begin{remark}
The algorithm to restore the monodromy zeta functions from the reduced equivariant
Poincar\'e series
$\rho P^{G}_{\{\nu_i\}}(\seq{t}{1}{r})$ includes the factorization of it into
the product of the binomials of the form $(1-\tt^{\ww})^{-s_{[H],\ww} [G/H]}$.
Without this factorization an algorithm is not clear.
\end{remark}

For an arbitrary faithful action of a finite group $G$ on $\C^2$, the equivariant
monodromy zeta functions $\zeta^{G}_{f}(G)$
and $\w{\zeta}^{G}_{f}(G)$ contains some factors of the form
$(1-t^k)^{-s [G/(e)]}$ and also factors of the form
$(1-t^k)^{-[G/H]}$ corresponding to lines consisting of points
with a non-trivial isotropy subgroup $H$. (We shall call these
lines {\it exceptional} ones.) If the representation of $G$ on
$\C^2$
is known, one knows all these possible factors and the problem to restore
the monodromy zeta functions from the equivariant Poincar\'e
series  is somewhat simpler. Therefore we shall not assume the
representation of $G$ on $\C^2$ to be known.

One can see that to make it possible to restore the equivariant
zeta functions from the equivariant Poincar\'e series
$P^{G}_{\{\nu_i\}}(\tt)$ one has to exclude certain cases (cf.
with Theorem 3 in \cite{Arkiv}).

\begin{example}
Let us consider two actions of the cyclic group $\Z_{6}$ on
$\C^2$ by representations
$\sigma*_1(x,y) = (\sigma^2 x , \sigma y)$ and
$\sigma*_2(x,y) = (\sigma^3 x , \sigma y)$ respectively
($\sigma = \exp(2\pi i /6)$ is the generator of $\Z_6$) and let
$C$ be the curve $\{x^6=0\}$. One has
$P^{\Z_6}_{\{\nu_i\}}(\tt)= (1-t_1\cdot\ldots\cdot
t_{6})^{-[\Z_6/\Z_6]_{\sigma}}$, where $\sigma$ is the
one-dimensional representation defined by $\sigma*z=\sigma z$.
Moreover the equivariant zeta functions
$\zeta^{G}_f(t)$  and $\w\zeta^{G}_f(t)$ are different in these
two cases, they are of the form
$(1-t^k)^{- s [\Z_6/\Z_2]}$ and to
$(1-t^k)^{- s [\Z_6/\Z_3]}$, $s\neq 0$, respectively. Thus they
are not determined by the equivariant Poincar\'e series
$P^{\Z_6}_{\{\nu_i\}}(\tt)$.
\end{example}

\begin{theorem}
Let $\C^2$ be endowed with a faithful action (a representation)
of a finite group $G$ and let $C=\{f=0\} = \bigcup_{i=1}^r C_i$,
where some of the components $C_i$ may coincide, be the zero
level curve of a $G$-invariant function germ
$f:(\C^2,0)\to(\C,0)$.
Assume that the curve $C$ does not contain a smooth curve
invariant with respect to a non-trivial element of $G$ whose
action on $\C^2$ is not a scalar one.
Let $\{\nu_i\}_{i=1}^{r}$ be the corresponding collection of
valuations.
Then the $G$-equivariant Poincar\'e series
$P^{G}_{\{\nu_i\}}(\seq t1r)$ determines the equivariant zeta
functions
$\zeta^{G}_f(t)$  and $\w\zeta^{G}_f(t)$.
\end{theorem}

\begin{proof}
Let
$$
P^{G}_{\{\nu_i\}}(\tt) = \prod_{\ww, [H], \alpha}
(1-\tt^{\ww})^{- s_{[H], \ww, \alpha} [G/H]_{\alpha}}
$$
be the A'Campo type factorization of the equivariant Poincar\'e
series (here $[H]$ runs over the conjugancy classes of subgroups
of $G$, $\alpha$ is a one-dimensional representation of $H$).
In order to restore the zeta functions
$\zeta^{G}_f(t)$  and $\w\zeta^{G}_f(t)$ one has to ``localize"
the factors corresponding to the exceptional lines and to
determine the action of the corresponding isotropy subgroups
(the subgroups preserving the lines) on them.

First we shall consider the case of an Abelian group $G$. This
makes the idea of the proof more transparent and permits to
describe the general case in a shorter way.

A representation of an Abelian group $G$ splits into the direct
sum of two one-dimensional representations. The action of $G$ can
be not free on some of the corresponding coordinate lines (on
non of them, or on one of them, or on both of them). If the
action on a coordinate line is not free, this is an exceptional
one. According to the assumption of the Theorem these exceptional
lines are not contained in the curve $C$. In the equation
\ref{ACampo} these exceptional lines represent a zero-dimensional
strata of the stratification $\{\Xi\}$ and provide to
$P^{G}_{\{\nu_i\}}(\tt)$ factors of the form
$(1-\tt^{\ww})^{-[G/G]_{\alpha}}$, where $\alpha$ is a
one-dimensional representation of $G$. Moreover , this factor has
a minimal exponent $\ww$ among the factors of the form
$(1-\tt^{\ww'})^{-s_{G,\ww',\alpha} [G/G]_{\alpha}}$ with
positive $s_{G,\ww',\alpha}$. This gives the factor corresponding
to one of the axis. Let us consider factors of the form
$(1-\tt^{\ww'})^{-s [G/G]_{\beta}}$ with $\beta\neq \alpha$. If
such factors do not exist, the action of $G$ is a scalar one
(i.e. an element $a\in G$ acts on $\C^2$ by multiplication by
$\alpha(a)$). In this case the action of $G$ on
$\C^2\setminus\{0\}$ is free and thus the equivariant zeta
functions
$\zeta^{G}_f(t)$  and $\w\zeta^{G}_f(t)$ are restored from the
equivariant Poincar\'e series as above (in Proposition 1). If
factors of the form
$(1-\tt^{\ww'})^{-s [G/G]_{\beta}}$ with $\beta\neq \alpha$ exist,
then the minimal among them with $s$ positive (in fact with
$s=1$) corresponds to the other axis.
Moreover, the representation of $G$ on $\C^2$ is the direct sum
of the one-dimensional representations $\alpha$ and $\beta$.
The appearance of the corresponding factors in the equivariant
zeta functions depends on the representations $\alpha$ and
$\beta$. (If both of them are exact, the $G$-action on
$\C^2\setminus\{0\}$ is free and the way to restore the
equivariant zeta functions is as in Proposition 1.)

In general, to get the equivariant zeta function
$\zeta^{G}_f(t)$  from the equivariant
Poincar\'e series, one has to substitute the described factors
$(1-\tt^{\ww_1})^{-[G/G]_{\alpha}}$ and
$(1-\tt^{\ww_2})^{-[G/G]_{\beta}}$ by
$$
\left(1-t
^{\frac{\abs{\ww_1}}{\abs{G}}}\right)^{- \frac{\abs{\Ker
\beta}}{\abs G} [G/\Ker \beta]}
\qquad \mbox{and} \qquad
\left(1-t^{\frac{\abs{\ww_2}}{\abs{G}}}\right)^{- \frac{\abs{\Ker
\alpha}}{\abs G} [G/\Ker \alpha]}
$$
respectively. (Note that in some sense the representations
$\alpha$ and $\beta$ exchange their positions.) All other factors
have to be modified as in Proposition 1.

To get the equivariant zeta function
$\w \zeta^{G}_f(t)$,
one has to substitute these two factors by
$$
\left(1-t
^{\frac{\abs{\ww_1}\abs{\Ker
\beta}}{\abs{G}^2}}\right)^{
- [G/\Ker \beta]}
\qquad \mbox{and} \qquad
\left(1-t
^{\frac{\abs{\ww_2}\abs{\Ker
\alpha}}{\abs{G}^2}}\right)^{
- [G/\Ker \alpha]}
$$
respectively.

Now let $G$ be an arbitrary finite group acting faithfully on
$\C^2$. An exceptional line is the set of fixed points of an
Abelian subgroup $H\subset G$, $H\neq (e)$. The elements of the
normalizer $N_G(H)$ preserve the line. All other elements of $G$
send the line to other exceptional ones. The normalizer $N_G(H)$
is an Abelian group. Its representation on $\C^2$ splits into two
different one-dimensional representations. (If these two
representations coincide, the action of $N_G(H)$ is a scalar one
and thus $N_G(H)$, and therefore $H$, act freely on
$\C^2\setminus\{0\}$.) The corresponding coordinate lines provide
into the equation (\ref{ACampo}) the factors of the form
$(1-\tt^{\ww_1})^{-[G/N_G(H)]_{\alpha}}$ and
$(1-\tt^{\ww_2})^{-[G/N_G(H)]_{\beta}}$ which are localized in
the same way as above (for an Abelian $G$). This means that one
of them has a minimal exponent $\ww'$ among the factors of the
form
$(1-\tt^{\ww'})^{- s[G/N_G(H)]_{\gamma}}$ with positive $s$ and
the other one has a minimal exponent $\ww'$ among the similar
factors with a different representation.

Now, to get the equivariant zeta function
$\zeta^{G}_f(t)$  from the equivariant
Poincar\'e series, one has to substitute the obtained
factors
$(1-\tt^{\ww_1})^{-[G/N_G(H)]_{\alpha}}$ and
$(1-\tt^{\ww_2})^{-[G/N_G(H)]_{\beta}}$ by
$$
\left(1-t
^{\frac{\abs{\ww_1}}{\abs{G}}}\right)^{- \frac{\abs{\Ker
\beta}}{\abs{N_G(H)}} [G/\Ker \beta]}
\qquad \mbox{and} \qquad
\left(1-t^{\frac{\abs{\ww_2}}{\abs{G}}}\right)^{- \frac{\abs{\Ker
\alpha}}{\abs{N_G(H)}} [G/\Ker \alpha]}
$$
respectively.

To get the equivariant zeta function
$\w \zeta^{G}_f(t)$,
one has to substitute them by
$$
\left(1-t
^{\frac{\abs{\ww_1}\abs{\Ker
\beta}}{\abs{G}\abs{N_G(H)}}}\right)^{-
[G/\Ker \beta]}
\qquad \mbox{and} \qquad
\left(1-t^{\frac{\abs{\ww_2}\abs{\Ker
\alpha}}{\abs{G}\abs{N_G(H)}}}\right)^{-
[G/\Ker \alpha]}
$$
respectively.
\end{proof}

Addresses:

A. Campillo and F. Delgado:
IMUVA (Instituto de Investigaci\'on en
Matem\'aticas), Universidad de Valladolid.
Paseo de Bel\'en, 7. 47011 Valladolid, Spain.
\newline E-mail: campillo\symbol{'100}agt.uva.es, fdelgado\symbol{'100}agt.uva.es

S.M. Gusein-Zade:
Moscow State University, Faculty of Mathematics and Mechanics, Moscow, GSP-1, 119991, Russia.
\newline E-mail: sabir\symbol{'100}mccme.ru


\begin{thebibliography}{12}

\bibitem{IJM} Campillo A., Delgado F., Gusein-Zade S.M.
The Alexander polynomial of a plane curve singularity and
integrals with respect to the Euler characteristic. Internat. J.
Math., v.14, no.1, 47--54 (2003).

\bibitem{MMJ} Campillo A., Delgado F., Gusein-Zade S.M.
On Poincar\'e series of filtrations on equivariant functions of
two variables. Mosc. Math. J., v.7, no.2, 243--255 (2007).

\bibitem{RMC} Campillo A., Delgado F., Gusein-Zade S.M.
Equivariant Poincar\'e series of filtrations.
Rev. Mat. Complut., v.26, 241--251 (2013).

\bibitem{Arkiv} Campillo A., Delgado F., Gusein-Zade S.M.
Equivariant Poincar\'e series of filtrations and topology.
Ark. Mat., v.52, 43--59 (2014).

\bibitem{CDK} Campillo A., Delgado F., Kiyek K. Gorenstein
property and symmetry for one-dimensional local Cohen--Macaulay
rings. Manuscripta Mathematica, v.83, no.3--4, 405--423 (1994).

\bibitem{GLM} Gusein-Zade S.M., Luengo I., Melle-Hern\'andez A.
An equivariant version of the monodromy zeta function.
In: Geometry, Topology, and Mathematical Physics: S. P. Novikov's Seminar: 2006--2007
(V.M.Buchstaber and I.M.Krichever, eds.), AMS, 2008
(American Mathematical Society Translations: Series 2, vol.{\bf 224}), pp. 139--146.

\bibitem{GLM-zeta2} Gusein-Zade S.M., Luengo I., Melle-Hern\'andez A.
On an equivariant version of the zeta function of a transformation.
arXiv: 1203:3344.

\bibitem{Lambda}
Knutson D. $\lambda$-rings and the representation theory of the symmetric group.
Lecture Notes in Mathematics, vol.{\bf 308}. Springer-Verlag,
Berlin-New York, 1973.

\bibitem{stanley} Stanley R.P. Enumerative combinatorics. Vol. 2. Cambridge Studies
in Advanced Mathematics {\bf 62}. 1999.

\bibitem{Nemethi} N\'emethi, A.
Poincar\'e series associated with surface singularities.
Singularities I, Contemp. Math., v.474,  271--297 (2008).

\bibitem{TtD} tom Dieck T. Transformation groups and representation theory.
Lecture Notes in Mathematics, vol.{\bf 766}. Springer, Berlin,
1979.

\bibitem{Wall} Wall C.T.C. Singular points of plane curves.
London Mathematical Society Student Texts, vol.{\bf 63}.
Cambridge University Press, Cambridge, 2004.

\bibitem{Ya} Yamamoto M. Classification of isolated algebraic singularities by
their Alexander polynomials. Topology, v.23, no.3, 277--287 (1984).

\end{thebibliography}
\end{document}